\documentclass[letterpaper, 10 pt, conference]{ieeeconf}  

\IEEEoverridecommandlockouts                              

\usepackage{color}
\usepackage{graphics} 
\usepackage{epsfig} 
\usepackage{times} 
\usepackage{amsmath} 
\usepackage{amssymb}  
\usepackage{lipsum}
\usepackage{multirow}
\usepackage[english]{babel}
\usepackage{tikz}
\usetikzlibrary{calc,patterns,angles,quotes}
\usetikzlibrary{arrows, automata}
\usepackage{tkz-euclide}
\usepackage{subcaption}

\newtheorem{theorem}{Theorem}[section]
\newtheorem{proposition}{Proposition}[section]
\newtheorem{definition}{Definition}[section]

\newtheorem{remarkth}[definition]{Remark}
\newenvironment{remark}{\begin{remarkth}\upshape}{\end{remarkth}}

\newcommand{\X}{\mathcal{X}}

\title{\LARGE \bf
The Bouncing Penny and Nonholonomic Impacts
}

\author{William Clark$^{1}$,  Anthony Bloch$^{1}$
\thanks{$^{1}$W. Clark, and A. Bloch are with the Department of Mathematics, University of Michigan, 530 Church St. Ann Arbor, 48109, Michigan, USA. Research
supported in part by the NSF and AFOSR.
        {\tt\small wiclark@umich.edu, abloch@umich.edu}}%
}

\begin{document}

\maketitle
\thispagestyle{empty}
\pagestyle{empty}

\begin{abstract}
The evolution of a Lagrangian mechanical system is variational. Likewise, when dealing with a hybrid Lagrangian system (a system with discontinuous impacts), the impacts can also be described by variations. These variational impacts are given by the so-called Weierstrass-Erdmann corner conditions. Therefore, hybrid Lagrangian systems can be completely understood by variational principles. 

Unlike typical (unconstrained / holonomic) Lagrangian systems, nonholonomically constrained Lagrangian systems \textit{are not} variational. However, by using the Lagrange-d'Alembert principle, nonholonomic systems can be described as projections of variational systems. This paper works out the analogous version of the Weierstrass-Erdmann corner conditions for nonholonomic systems and examines the billiard problem with a rolling disk.
\end{abstract}

\section{Introduction}
Mechanical systems exhibiting unilateral constraints have a rich history and are used to model many interesting and important systems, see \cite{brogliato1999nonsmooth} and the over 1000 references therein. A particular and well-studied example of these systems is the billiard problem. This is characterized by a particle on a table moving in straight lines and reflecting off of the boundary by specular reflection. Some examples where this is studied is \cite{Baryshnikov}, \cite{doi:10.1119/1.1738428}, \cite{FERES20041541}, \cite{artin1924mechanisches}, \cite{chernov2003introduction} and \S9.2 of \cite{katok1995introduction}. Moreover, this problem has even been applied to biological processes \cite{SPAGNOLIE201733}.

Contrary to the mathematical billiard problem, billiard balls are spherical and are capable of much more complex impacts than those following specular reflection. Cases where the geometry of the ball are taken into consideration can be found in \cite{coxbilliards} and \cite{BROOMHEAD1993188}. Even though these study impacts with spheres, the spheres are unconstrained before and after impact. This misses the nonholonomic constraint imposed by the ball rolling without slipping. Therefore, in order to fully encapsulate the behavior of billiards, both the geometry of the ball and the nonholonomic constraint need to be taken into account.

This paper explores and develops a framework for working out impacts of mechanical systems subjected to nonholonomic constraints. Nonholonomic impacts as described here fall into two categories: elastic and plastic. Elastic impacts are taken to be variational and thus follow from a direct application of the Weierstrass-Erdmann corner conditions (see \S 4.4 of \cite{kirk1970optimal} or \S 3.5 of \cite{brogliato1999nonsmooth}). On the other hand, plastic impacts are taken to be orthogonal (with respect to the kinetic energy metric) projections onto the constraint distribution.

To showcase the utility of nonholonomic impacts, this paper concludes by studying the billiard problem of the vertical rolling disk (henceforth referred to as a penny). This consists of the rolling penny (which possesses two nonholonomic constraints) freely moving on a table until it bounces off of the table edge. Numerical simulations compare plastic and elastic impacts for a penny on an elliptical table as well as hint at chaotic properties of the elastic impact.

This paper is laid out as follows: Section \ref{sec:prelim} outlines some preliminary notions of hybrid systems and geometric mechanics. Section \ref{sec:holonomic} reviews how to compute (variational) impacts for unconstrained Lagrangian systems and Section \ref{sec:nonholon} extends this to nonholonomic systems. Sections \ref{sec:penny} and \ref{sec:numerical} contain the derivation of the motion for the rolling penny billiard problem as well as some numerical simulations. Finally, section \ref{sec:future} contains possible future research directions.

\section{Preliminaries}\label{sec:prelim}
In what follows, we will define and study the notion of a \textit{nonholonomic hybrid dynamical system}. This will be done by joining the theory of hybrid dynamical systems and geometric mechanics. This section reviews some preliminary notions from these fields.
\subsection{Hybrid Systems}
Informally speaking, a hybrid systems is something that exhibits both continuous and discrete dynamics, see \cite{teel} for an extensive treatment. For our purposes, we define a hybrid dynamical system as well as the notion of an execution as follows, similar to \cite{1656623}.
\begin{definition}
	The 4-tuple, $(\X,S,f,\delta)$, is a hybrid dynamical system if:
	\begin{enumerate}
		\item $\X$ is a smooth manifold called the state-space,
		\item $S\subset\X$ is a codimension 1 embedded manifold called the impact surface,
		\item $f:\X\to T\X$ is a smooth vector field, and
		\item $\delta:S\to \X$ is a smooth map called the impact map.
	\end{enumerate}
\end{definition}
\begin{definition}
	An execution of $(\X,S,f,\delta)$ is a tuple $(\Lambda,J,\mathcal{C})$ such that
	\begin{enumerate}
		\item $\Lambda = \left\{ 0,1,2,\ldots\right\}\subset\mathbb{N}$ is a finite or infinite indexing set,
		\item $J = \left\{ I_i\right\}_{i\in\Lambda}$ is a collection of intervals where if $i,i+1\in\Lambda$ then $I_i = [\tau_i,\tau_{i+1}]$ and if $|\Lambda|=N<\infty$ then $I_{N-1} = [\tau_{N-1},\tau_N]$ or $[\tau_{N-1},\tau_N)$ or $[\tau_{N-1},\infty)$,
		\item $\mathcal{C} = \left\{c_i\right\}_{i\in\Lambda}$ is a family of functions, $c_i:I_i\to \X$ such that the following three conditions hold: (i) $\dot{c}_i(t) = f(c_i(t))$, (ii) $c_i(\tau_{i+1})\in S$, and (iii) $\delta(c_i(\tau_{i+1})) = c_{i+1}(\tau_{i+1})$.
	\end{enumerate}
\end{definition}
\begin{remark}
	Through out the remainder of this paper, an execution $(\Lambda,J,\mathcal{C})$ will simply be referred to as $c:\cup I_i\to \X$ where $c(t) = c_i(t)$ when $t\in I_i$. The break points will be denoted as $c^+(\tau_{i+1}) = c_{i+1}(\tau_{i+1})$ and $c^-(\tau_{i+1}) = c_i(\tau_{i+1})$.
\end{remark}
\begin{remark}
	A common issue when studying hybrid systems is the notion of a Zeno solution. This happens when $|\Lambda|=\infty$ but $\lim_{i\to\infty} \tau_i < \infty$.  However, we will not be concerned about this behavior here because the purpose of this paper is to derive the impact map, $\delta$, for nonholonomic systems. See \cite{1656623} for more on Zeno solutions.
\end{remark}
\subsection{Geometric Mechanics}
Geometric mechanics is generally given in the Lagrangian or Hamiltonian formalism. Here we will only be interested in the Lagrangian picture but everything transfers to the Hamiltonian side assuming mild conditions.

Lagrangian mechanics consists of two pieces of information, a smooth manifold $Q$ called the configuration space and a smooth function $L:TQ\to\mathbb{R}$ called the Lagrangian function. By imposing additional assumptions on the Lagrangian, $L$, we introduce various classes of Lagrangians (see \cite{abraham2008foundations} and \cite{arnoldmechanics} for background). Before we can define these classes, we first need the notion of the fiber derivative.
\begin{definition}
	Let $Q$ be a manifold and $L\in C^\infty(TQ)$. Then the map $\textbf{F}L:TQ\to T^*Q$ with $v_q\mapsto DL_q(v_q)\in T_q^*Q$ is called the fiber derivative. Here, $L_q$ denotes the restriction of $L$ to the fiber over $q\in Q$.
\end{definition}
We now introduce three classes of Lagrangians.
\begin{definition}
	The Lagrangian is called:
	\begin{enumerate}
		\item regular if $\textbf{F}L$ is nondegenerate at each $q\in Q$,
		\item hyperregular if $\textbf{F}L:TQ\to T^*Q$ is a global diffeomorphism, and
		\item natural if $L(q,v) = 1/2 g_q(v,v)-V(q)$ where $(Q,g)$ forms a Riemannian manifold, i.e. $L$ is of the form kinetic energy minus potential energy.
	\end{enumerate}
\end{definition}
Notice that natural implies hyperregular which implies regular. Through out the rest of this work, all Lagrangians will be assumed to be either hyperregular or natural. When the Lagrangian is natural, the fiber derivative agrees with the musical isomorphisms, $\sharp$ and $\flat$.
\begin{equation}
\begin{split}
\flat:TQ\to T^*Q,&\quad v\mapsto g(v,\cdot),\\
\sharp:T^*Q\to TQ,& \quad \sharp = \flat^{-1}.
\end{split}
\end{equation}

Nonholonomic Lagrangian systems are Lagrangian systems with a constraint distribution, $\Delta$. The constraint distribution, $\Delta\subset TQ$, satisfies: $\Delta_x\subset T_xQ$ is a subspace. A distribution being the same dimension everywhere is called regular. This leads to the definition of a nonholonomic hybrid Lagrangian.

\begin{definition}
	The 4-tuple $(Q,L,\Delta,S)$ is a \textit{nonholonomic hybrid Lagrangian} if
	\begin{enumerate}
		\item $Q$ is a smooth manifold,
		\item $L:TQ\to\mathbb{R}$ is a smooth Lagrangian,
		\item $\Delta \subset TQ$ is a smooth regular distribution, called the constraint distribution, and
		\item $S\subset Q$ is a smooth embedded codimension 1 submanifold, called the impact surface.
	\end{enumerate}
\end{definition}
Notice that a nonholonomic hybrid Lagrangian, $(Q,L,\Delta,S)$, can be made into a hybrid dynamical system with $\X=\Delta$, $S=S$, $f$ is the equations of motion for the nonholonomic Lagrangian, and $\delta$ is the object of study in Section \ref{sec:nonholon}.
\section{Holonomic Systems}\label{sec:holonomic}
In this section, we restrict ourselves to the special case where $\Delta = TQ$. This class of systems is almost the same as studied in \cite{lamperski2008sufficient} and \cite{1656623}. The difference is that there the impact surface is given by the level set of a smooth function $h:Q\to\mathbb{R}$. However, since $S$ is embedded, we can \textit{locally} describe $S$ by the level set of a function. All of the following computations will be done in local coordinates, so this will never be an issue. For the remainder of this paper, we will view $S$ as locally the level set of $h:Q\to\mathbb{R}$.

Due to the fact that we are studying a hybrid system, we need to define both the continuous and impact dynamics for $(Q,L,S)$ ($\Delta$ is suppressed here because it is taken to be $TQ$ and offers no new information). The continuous dynamics are governed by the Euler-Lagrange equations:
\begin{equation}\label{eq:contL}
\frac{d}{dt}\frac{\partial L}{\partial \dot{q}^i} - \frac{\partial L}{\partial q^i} = 0.
\end{equation}
The Euler-Lagrange equations are \textit{variational}. With this insight, we define the impact map to be variational as well. This is realized by the Weierstrass-Erdmann corner conditions (see \S 4.4 of \cite{kirk1970optimal} or \S 3.5 of \cite{brogliato1999nonsmooth}):
\begin{equation}\label{eq:WEL}
\begin{split}
\textbf{F}L^+ - \textbf{F}L^- &= \alpha \cdot dh,\\
L^+ - \langle \textbf{F}L^+,\dot{q}^+\rangle &= L^- - \langle \textbf{F}L^-,\dot{q}^-\rangle,
\end{split}
\end{equation}
where $S = \left\{ q\in Q : h(q) = 0\right\}$ and the multiplier $\alpha$ is set so both equations are satisfied. These corner conditions have a clearer interpretation in the Hamiltonian setting:
\begin{equation}\label{eq:pH}
\begin{split}
p^+ &= p^- + \alpha\cdot dh, \\
H^+ &= H^-.
\end{split}
\end{equation}
i.e. energy at impacts is conserved and the change in momentum is perpendicular to the impact surface which is precisely specular reflection. These equations can be explicitly solved when the system is natural. Recall that $\nabla h = dh^\sharp$ and $p = \textbf{F}L(\dot{q}) = \dot{q}^\flat$.
\begin{theorem}
	Assuming that $L(q,v) = 1/2g_q(v,v)-V(q)$ is natural, the impact map is given by $(q^-,\dot{q}^-) \mapsto (q^-,P(q^-,\dot{q}^-))$ where
	\begin{equation}\label{eq:impact}
	P(q,\dot{q}) = \dot{q} - 2 \frac{dh(\dot{q})}{g(\nabla h,\nabla h)} \nabla h.
	\end{equation}
\end{theorem}
Notice that \eqref{eq:impact} matches the impact equation in \cite{1656623} with the exception of a coefficient of restitution $0\leq e \leq 1$. That is,
\begin{equation}\label{eq:restit}
\dot{q}^+ = \dot{q}^- - (1+e)\frac{dh(\dot{q}^-)}{g(\nabla h,\nabla h)} \nabla h.
\end{equation}
A hybrid Lagrangian systems is completely described by \eqref{eq:contL} and \eqref{eq:impact}.
\begin{remark}
	In equation \eqref{eq:restit}, the case when $e=1$ reduces to the variational case (elastic). When $e=0$, $P$ becomes the orthogonal projection (with respect to $g$) $TQ\to\ker dh$ (plastic).
\end{remark}
\section{Nonholonomic Systems}\label{sec:nonholon}
We now consider the general case of $(Q,L,\Delta,S)$ where $\Delta \ne TQ$. The continuous and impact dynamics need to be modified in the presence of these constraints. For this section, Einstein summation notation will be used: when an expression has a matching lower and upper index, a sum is implied, i.e. $a_kb^k = \sum_k\, a_kb^k$.
\subsection{Continuous}
We make a brief review of nonholonomic equations of motions. For a more in-depth treatment see \cite{bloch2008nonholonomic}.

Let $\Delta\subset TQ$ be a regular smooth distribution. This can be realized as the kernel of a collection of linearly independent 1-forms:
\begin{equation}\label{eq:ker}
\Delta_q = \bigcap_{k=1}^m \, \ker\omega^k_q.
\end{equation}
Then the collection of 1-forms, $\{\omega^k\}$, is called the constraining 1-forms. In local coordinates we can write each constraining 1-form as
\begin{equation}
\omega^k(\dot{q}) = a^k_j(q)\dot{q}^j = 0,\quad k=1,\ldots,m.
\end{equation}
With this, we can determine the nonholonomic equations of motion by the Lagrange-d'Alembert principle. In short, this principle states that in order to impose the constraints to the Lagrangian system, the constraining forces do no work. This is formally defined below.
\begin{definition}\label{def:LD}
	The principle
	\begin{equation}
	\delta \int_a^b \, L(q(t),\dot{q}(t))\, dt = 0,
	\end{equation}
	where the virtual displacements $\delta q$ are assumed to satisfy the constraints:
	\begin{equation}
	a_j^k(q)\delta q^j = 0, \quad k = 1,\ldots, m,
	\end{equation}
	is called the \textit{Lagrange-d'Alembert principle}.
\end{definition}
The equations of motion generated by definition \ref{def:LD} will henceforth be referred to as the \textit{dynamic nonholonomic equations of motion}.
\begin{remark}
  Notice that the dynamic nonholonomic equations are different from  variational nonholonomic equations or so-called vakonomic systems. See \cite{arnoldIII}
  and \cite{bloch2008nonholonomic}
  for a comparison between these formulations.  However, the dynamic nonholonomic equations give the \textit{correct} mechanical equations of motion. 
\end{remark}

We can use the Lagrange-d'Alembert principle to write down the dynamic nonholonomic equations of motion (which will look like a modified version of the Euler-Lagrange equations). Using the fact that the $m$ 1-forms are linearly independent, we can choose local coordinates such the forms have the form
\begin{equation}
\omega^a_q = ds^a + A_\alpha^a(r,s)dr^\alpha,\quad a = 1,\ldots,m,
\end{equation}
where $q = (r,s)\in \mathbb{R}^{n-m}\times\mathbb{R}^m$ are local coordinates. In these coordinates, the equations of motion can be expressed as $n-m$ second-order differential equations and $m$ first-order constraint equations.
\begin{equation}\label{eq:non_2nd}
\begin{split}
\left( \frac{d}{dt}\frac{\partial L}{\partial \dot{r}^\alpha} - \frac{\partial L}{\partial r^\alpha} \right) &= A_\alpha^a \left( \frac{d}{dt} \frac{\partial L}{\partial \dot{s}^a} - \frac{\partial L}{\partial s^a} \right),\\
& \alpha = 1,\ldots, n-m,
\end{split}
\end{equation}
\begin{equation}\label{eq:non:1st}
\dot{s}^a = -A_\alpha^a \dot{r}^\alpha,\quad a = 1,\ldots,m.
\end{equation}
\subsection{Impacts}
The study of constrained impacts breaks into two categories: elastic and plastic. Elastic impacts arise variationally while plastic impacts result from a projection, essentially the $e=1$ and $e=0$ cases in equation \eqref{eq:restit}.
\subsubsection{Elastic Impacts}
Consider the variational impact equations:
\begin{equation}
\begin{split}
\left( \textbf{F}L^+ - \textbf{F}L^- \right) \delta q &= 0, \\
\left(H^+ - H^- \right) \delta t &= 0.
\end{split}
\end{equation}
If the impact time is free, i.e. $\delta t \ne 0$, then this leads to conservation of energy of the impacts which is the second half of equation \eqref{eq:WEL}. The spatial variations, $\delta q$, have to satisfy two sets of constraints: the nonholonomic constraints $\omega^k(\delta q) = 0$, and the impact constraint $dh(\delta q) = 0$. This leads to the following system of equations,
\begin{equation}\label{eq:var_impact}
\begin{split}
\Delta \textbf{F}L &= \lambda_k \omega^k + \alpha\cdot dh, \\
\Delta H &= 0,\\
\omega^k(\dot{q}^+) &= 0.
\end{split}
\end{equation}
\begin{proposition}\label{prop:uniq}
	Assuming $L$ is a natural Lagrangian and all $\omega^k$ and $dh$ are linearly independent, there exists at most an unique nontrivial solution to \eqref{eq:var_impact}.
\end{proposition}
\begin{proof}
  Using the fact that $p^+ = p^- + \lambda_k\omega^k+\alpha\cdot dh$ and
  substituting this into the other constraints, we get $m$ linear equations and one quadratic equation.
	\begin{equation}\label{eq:linear}
	\lambda_k \langle \omega^k,\omega^\ell\rangle + \alpha \langle dh,\omega^\ell \rangle = 0, \quad \forall \ell,
	\end{equation}
	\begin{equation}\label{eq:quad}\begin{split}
	2\lambda_k \langle p^-,\omega^k\rangle &+ 2\alpha\langle p^-,dh\rangle + \lambda_i\lambda_j\langle \omega^i,\omega^j\rangle \\
	&+ 2\lambda_k\alpha\langle \omega^k,dh\rangle + \alpha^2\langle dh,dh\rangle = 0
	\end{split}\end{equation}
	Equation \eqref{eq:linear} can be solved such that $\lambda_k = \lambda_k(\alpha)$ depend linearly on $\alpha$. When this is substituted into \eqref{eq:quad}, the resulting equation is quadratic in $\alpha$ with zero as a solution. Therefore it can have at most one nonzero answer.
\end{proof}
\begin{definition}
	Let $(Q,g)$ be a Riemannian manifold and $(Q,L,\Delta,S)$ be a corresponding natural nonholonomic Lagrangian. The impact map given by Proposition \ref{prop:uniq} is called the \textit{elastic nonholonomic impact map}.
\end{definition}
\subsubsection{Plastic Impacts}
While elastic impacts arise variationaly, plastic impacts come from projections. Given the unconstrained impact, $\dot{q}^+ = P(q,\dot{q})$, it is not generally true that $P(q,\dot{q})\in\Delta$. In order to enforce the constraints on $\dot{q}^+$, we orthogonally project $P$ onto the subspace $\Delta$. This leads to the following definition.
\begin{definition}
	Let $(Q,g)$ be a Riemannian manifold and $(Q,L,\Delta,S)$ be a corresponding natural nonholonomic Lagrangian. Denote $\pi_\Delta:TM\to\Delta$ as the orthogonal (with respect to $g$) projection onto the subspace $\Delta$. Then, the \textit{plastic nonholonomic impact map} is given by
	\begin{equation}
	(q,\dot{q})\mapsto \left(q,\pi_\Delta\circ P(q,\dot{q})\right),
	\end{equation}
	where $P$ is the variational impact in \eqref{eq:impact}.
\end{definition}
\begin{remark}
	Plastic impacts are commonly used for walking robots where the foot is not allowed to bounce off of the ground, see \cite{coleman1997motions} and \cite{garcia1998simplest}.
\end{remark}

When the underlying Lagrangian is natural, we can determine an explicit formula for $\pi_\Delta\circ P$.
\begin{proposition}
	Let $\Delta$ be defined as \eqref{eq:ker} and $W^k = \left(\omega^k\right)^\sharp$ be vector fields corresponding to the constraining 1-forms. Define the matrices
	\begin{equation}
	a^{ij} = \omega^i\left(W^j\right),\quad \left(a_{ij}\right) = \left(a^{ij}\right)^{-1}.
	\end{equation}
	Then the nonholonomic impact map is given by
	\begin{equation}\label{eq:nonhimpact}
	\dot{q} \mapsto \dot{q} - 2\frac{dh(\dot{q})}{g(\nabla h,\nabla h)} \left( 
	\nabla h - a_{ij}\omega^i(\nabla h)W^j \right).
	\end{equation}
\end{proposition}
\begin{proof}
	This follows from the fact that the orthogonal projection is given by
	\begin{equation}
	\pi_\Delta(\dot{q}) = \dot{q} - a_{ij}\omega^i(\dot{q})W^j,
	\end{equation}
	and that $\pi_\Delta(\dot{q}) = \dot{q}$ since the constraints are assumed to be satisfied before impact.
\end{proof}

\begin{remark}
	It is interesting to see the differences between plastic and elastic impacts. Elastic impacts result from solving all equations in \eqref{eq:var_impact} \textit{simultaneously}. On the other hand, plastic impacts result from first solving the two equations $\Delta \textbf{F}L = \alpha\cdot dh$ and $\Delta H = 0$. Once the multiplier $\alpha$ is found, then the $\lambda_k$ are found such that the constraints are satisfied.
\end{remark}
\section{Example: The vertical rolling disk}\label{sec:penny}
The vertical rolling disk is a simple example for illustrating nonholonomic dynamics. We will work with the somewhat nonphysical case where the disk is not permitted to tilt.
The configuration space and local coordinates for the rolling disk are given by $q = (x,y,\theta,\varphi)\in Q = \mathbb{R}^2\times S^1\times S^1$, denoting the position of the contact point, the rotation angle of the disk, and the orientation of the disk, respectively. 
\subsection{Dynamic nonholonomic equations}
The Lagrangian for the vertical disk is taken to be the kinetic energy (no potential force will be included) i.e.
\begin{equation}
L = \frac{1}{2}m\left( \dot{x}^2+\dot{y}^2 \right) + \frac{1}{2}I\dot{\theta}^2 + \frac{1}{2}J\dot{\varphi}^2.
\end{equation}
Here, $m$ is the mass of the disk, $I$ is the moment of inertia of the disk about the axis perpendicular to the plane of the disk, and $J$ is the moment of inertia about an axis in the plane of the disk.

If $R >0$ is the radius of the disk, the nonholonomic constraints for rolling \textit{without} slipping are
\begin{equation}\begin{split}
\dot{x} &= R\dot{\theta}\cos\varphi, \\
\dot{y} &= R\dot{\theta}\sin\varphi.
\end{split}\end{equation}
This can be expressed as $\Delta = \ker\omega^1\cap\ker\omega^2$ where
\begin{equation}
\begin{split}
\omega^1 &= dx - R\left(\cos\varphi\right) d\theta,\\
\omega^2 &= dy - R\left(\sin\varphi\right) d\theta.
\end{split}
\end{equation}
The equations of motions for this (uncontrolled) system are given by two dynamic equations and the two constraining equations (see \S 1.4 of \cite{bloch2008nonholonomic} for the derivation).
\begin{equation}\label{eq:penny_motion}
\begin{split}
J\ddot{\varphi} &= 0, \\
\left(I+mR^2\right)\ddot{\theta} &= 0,\\
\dot{x} &= R\dot{\theta}\cos\varphi,\\
\dot{y} &= R\dot{\theta}\sin\varphi. 
\end{split}
\end{equation}
These equations can be easily integrated. Let the initial conditions be $(x_0,y_0,\theta_0,\varphi_0)$ and call $\omega = \dot{\varphi}$, $\Omega = \dot{\theta}$, which are constants. Then, the (continuous) equations of motion are
\begin{equation}
\begin{split}
\varphi &= \omega t + \varphi_0, \\
\theta &= \Omega t + \theta_0, \\
x &= \frac{\Omega}{\omega} R \sin\left(\omega t + \varphi_0\right) - \frac{\Omega}{\omega}R\sin\varphi_0 + x_0,\\
y &= -\frac{\Omega}{\omega} R \cos\left( \omega t + \varphi_0\right) + \frac{\Omega}{\omega}R\cos\varphi_0 + y_0.
\end{split}
\end{equation}
\subsection{Impact map}
In this subsection, due to page constraints, we only derive the plastic nonholonomic impact map. The elastic nonholonomic impact map used in Section \ref{sec:numerical} comes from solving the system \eqref{eq:var_impact}.

Consider the case where the rolling disk is constrained to a pool table, so the disk can strike the edge of the table top and ``bounce'' off. Assume that the location of the wall around the table is given by $\tilde{S} = \left\{ (x,y)\in\mathbb{R}^2 : h(x,y) = 0 \right\}$ for some smooth function $h:\mathbb{R}^2\to\mathbb{R}$. Then an impact occurs when
\begin{equation}
\begin{bmatrix}
x + R\cos\varphi \\ y + R\sin\varphi
\end{bmatrix} \in \tilde{S}.
\end{equation}
Using \eqref{eq:impact} to determine the (pre-) impact map, we see that
\begin{equation}\label{eq:var_penny}
\begin{split}
\dot{x}^+ &= \dot{x}^- + \frac{C}{m}\frac{\partial h}{\partial x}, \\
\dot{y}^+ &= \dot{y}^- + \frac{C}{m}\frac{\partial h}{\partial y}, \\
\dot{\theta}^+ &= \dot{\theta}^-, \\
\dot{\varphi}^+ &= \dot{\varphi}^- + \frac{C}{J} R\left(\cos\varphi \frac{\partial h}{\partial y} - \sin\varphi \frac{\partial h}{\partial x} \right). 
\end{split}
\end{equation}
Here, the number $C=C(\varphi,\dot{x},\dot{y},\dot{\varphi})$ has the value
\begin{equation}
C = \frac{ -2\left[
	h_x\dot{x} + h_y \dot{y} 
	+ R\left( h_y \cos\varphi - h_x \sin\varphi \right) \dot{\varphi} \right]}
{ \frac{1}{m}\left[  h_x^2 + 
	h_y^2\right] + \frac{1}{J}R^2 
	\left( h_y \cos\varphi - h_x \sin\varphi \right)^2 },
\end{equation}
where $h_x$ and $h_y$ are the $x$ and $y$ partial derivatives of $h$. 
Notice that at impact the rotation angle of the disk, $\dot{\theta}$, is unchanged. This is because the constraint of no sliding has not yet been imposed. In order to apply the constraints, we compute $\pi_\Delta$. The corresponding vector fields, $W^i$, and multiplier matrix, $(a_{ij})$ are
\begin{equation}
\begin{split}
W^1 &= \frac{1}{m}\frac{\partial}{\partial x} - \frac{R}{I}\cos\varphi \frac{\partial}{\partial \theta}, \\
W^2 &= \frac{1}{m}\frac{\partial}{\partial y} - \frac{R}{I}\sin\varphi \frac{\partial}{\partial \theta},
\end{split}
\end{equation}
and
\begin{equation}\begin{split}
(a_{ij}) &= \begin{bmatrix}
m - K\cos^2\varphi & -K\sin\varphi\cos\varphi \\
-K\sin\varphi\cos\varphi & m-K\sin^2\varphi
\end{bmatrix},\\
K &= \frac{m^2R^2}{I+mR^2}.
\end{split}\end{equation}
We can now write $\pi_\Delta$ as a matrix with coordinates $(\dot{x},\dot{y},\dot{\theta},\dot{\varphi})$:

\begin{equation}\label{eq:pi_delta}
(I+mR^2)\cdot\pi_\Delta = \begin{bmatrix}
A & B \\ C & D
\end{bmatrix},
\end{equation}
with each block being
\begin{equation}
\begin{split}
A &= mR^2\begin{bmatrix}
\cos^2\varphi &\sin\varphi\cos\varphi \\
\sin\varphi\cos\varphi & \sin^2\varphi
\end{bmatrix}, \\
B &= IR\begin{bmatrix}
\cos\varphi & 0 \\
\sin\varphi & 0
\end{bmatrix}, \quad
C = mR\begin{bmatrix}
\cos\varphi & \sin\varphi \\
0 & 0
\end{bmatrix},\\
D &= \begin{bmatrix}
I & 0 \\ 0 & I + mR^2
\end{bmatrix}.
\end{split}
\end{equation}
The impact map is then given by composing \eqref{eq:pi_delta} with \eqref{eq:var_penny}.
There is one last thing to discuss before we move onto simulations of this model: what if the back end of the coin hits the surface instead of the front? i.e.
\begin{equation}
\begin{bmatrix}
x - R\cos\varphi \\ y - R\sin\varphi
\end{bmatrix} \in \tilde{S}.
\end{equation}
Then the same procedure as above is carried out. 
\section{Numerical Results}\label{sec:numerical}
Assume that the table-top is elliptical. i.e.
\begin{equation}
h(x,y) = \frac{x^2}{a^2} + \frac{y^2}{b^2} - 1.
\end{equation}
We will also assume that the disk is homogeneous and thin so $I = 1/2mR^2$ and $J=1/4mR^2$. For the remaining parameters, we would like $R<<a,b$ so the coin has ample room to explore. The values are in the table below and are chosen to be similar to a US penny. 
\begin{table}[h!]
	\centering
	\begin{tabular}{ c | c }
		Parameter & Value \\ \hline
		$R$ & 0.01 m\\  
		$m$ & 0.0025 kg  \\
		$I$ & $1.25\cdot 10^{-7}$ kg m$^2$ \\
		$J$ & $6.25\cdot 10^{-8}$ kg m$^2$  \\
		$a$ & 0.15 m or 0.20 m\\
		$b$ & 0.20 m
	\end{tabular}
	\caption{The chosen set of parameters for the time step example. $a$ takes the first value for the elliptical cases and the latter for the circular ones.}
	\label{tab:penny_param}
\end{table}
\subsection{Plastic vs Elastic}\label{subsec:pvse}
For the simulations, Figures \ref{fig:elastic_circle}, \ref{fig:plastic_circle}, \ref{fig:elastic_ellipse}, and \ref{fig:plastic_ellipse}, the same initial conditions are chosen: $x_0=y_0=\theta_0=0$, $\varphi_0 = \pi/2$, $\dot{\theta}_0 = 10$, and $\dot{\varphi}_0 = 0.2$.
\begin{figure}[h!]
	\centering
	\includegraphics[scale=0.37]{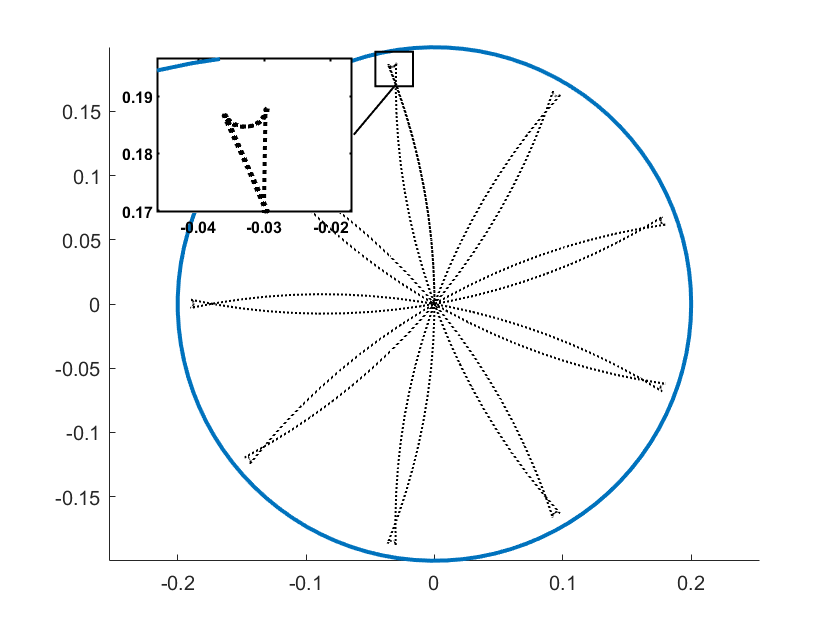}
	\caption{The first 20 impacts for the elastic impact on a circular table.}
	\label{fig:elastic_circle}
\end{figure}
\begin{figure}[h!]
	\centering
	\includegraphics[scale=0.37]{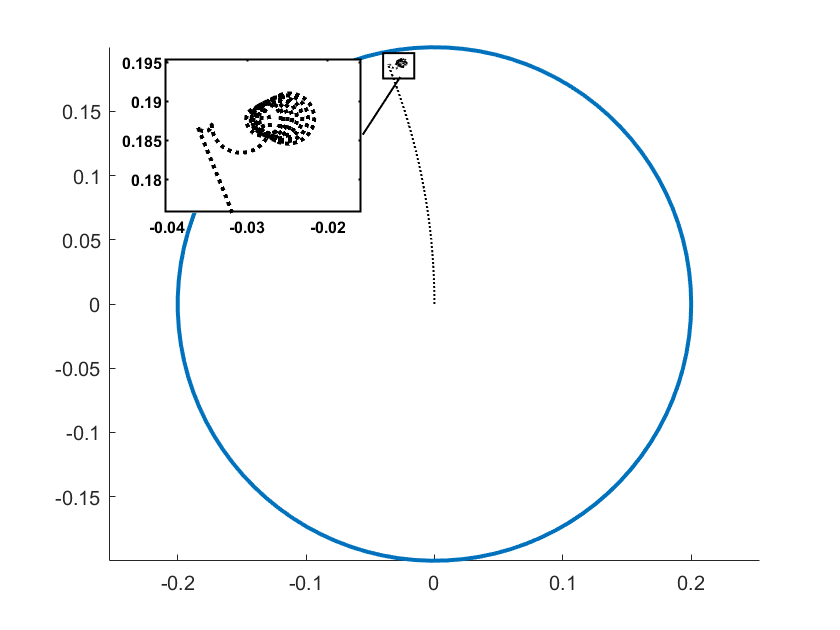}
	\caption{The first 20 impacts for the plastic impact on a circular table.}
	\label{fig:plastic_circle}
\end{figure}
\begin{figure}[h!]
	\centering
	\includegraphics[scale=0.37]{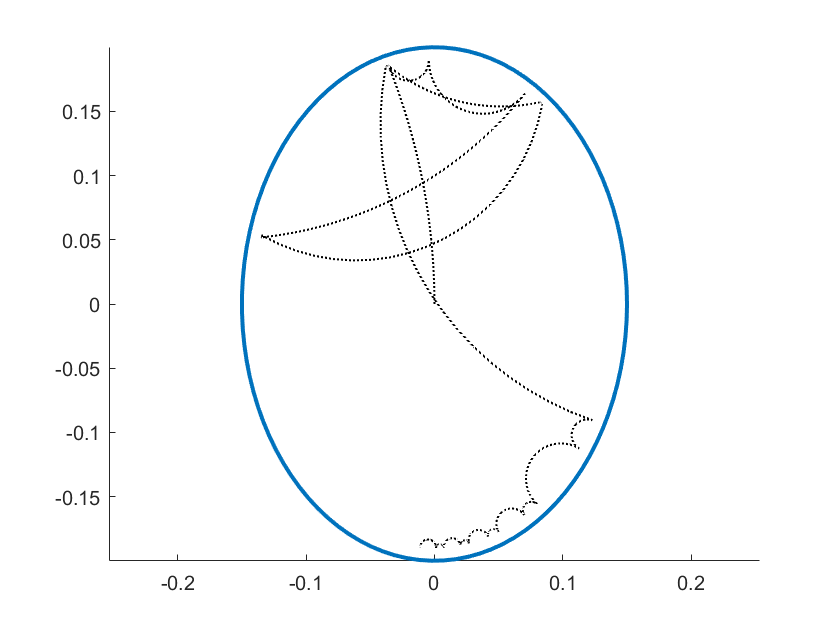}
	\caption{The first 20 impacts for the elastic impact on an elliptical table.}
	\label{fig:elastic_ellipse}
\end{figure}
\begin{figure}[h!]
	\centering
	\includegraphics[scale=0.37]{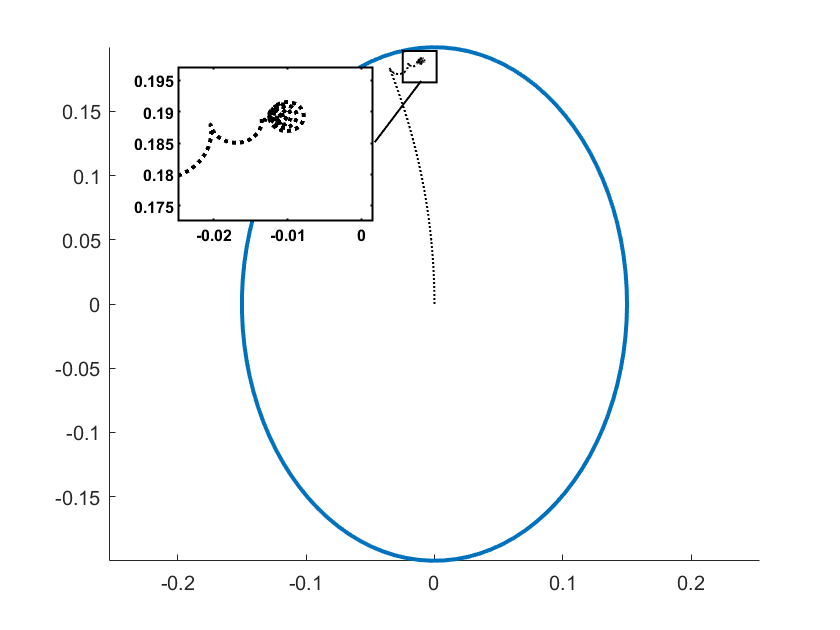}
	\caption{The first 20 impacts for the plastic impact on an elliptical table.}
	\label{fig:plastic_ellipse}
\end{figure}
\subsection{Hints of Chaos}
In addition to comparing trajectories of the elastic and plastic impacts, we compare how changes in initial conditions propagate with time. All 100 initial conditions are taken to be those chosen in \S\ref{subsec:pvse} except that $\dot{\theta}_0$ and $\dot{\varphi}_0$ are randomly perturbed by $<0.005$. These results are shown in Figure \ref{fig:chaos}.
\begin{figure}
	\begin{subfigure}[t]{.48\columnwidth}
		\centering
		\includegraphics[width=\linewidth]{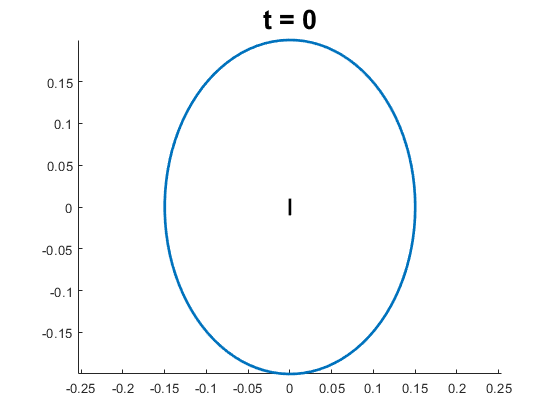}
	\end{subfigure}
	\hfill
	\begin{subfigure}[t]{.48\columnwidth}
		\centering
		\includegraphics[width=\linewidth]{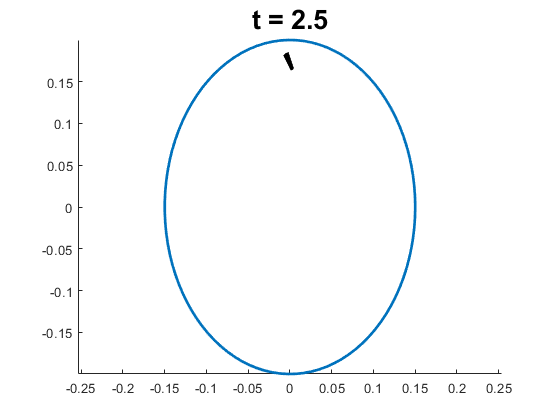}
	\end{subfigure}
	
	\medskip
	
	\begin{subfigure}[t]{.48\columnwidth}
		\centering
		\includegraphics[width=\linewidth]{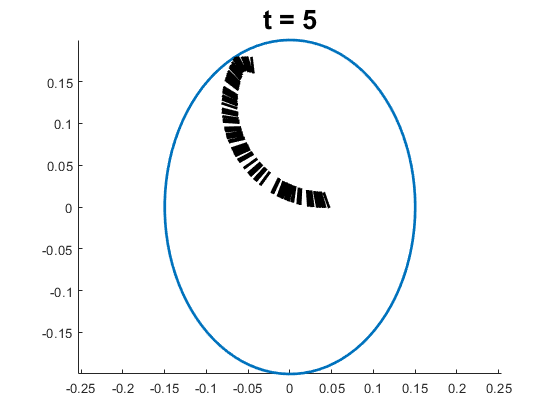}
	\end{subfigure}
	\hfill
	\begin{subfigure}[t]{.48\columnwidth}
		\centering
		\includegraphics[width=\linewidth]{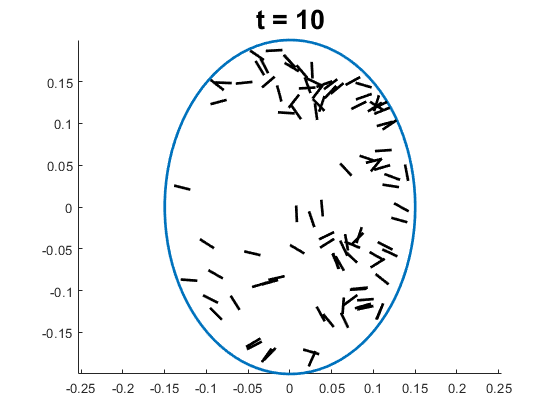}
	\end{subfigure}
	\caption{Plots of 100 different initial conditions at various times.}
	\label{fig:chaos}
\end{figure}
\section{Conclusions, Control and Future Work}\label{sec:future}
There are two key future directions for this research. The first is controlling hybrid nonholonomic systems. 
All of the motions considered here have been uncontrolled. However, this can be changed by modifying \eqref{eq:penny_motion} to $J\ddot{\varphi} = u_1$ and $(I+mR^2)\ddot{\theta} = u_2$ where $u_1$ and $u_2$ are controls. It is an object of future work to understand controllability of the impacting penny (it is known that the non-hybrid penny is controllable). For example one may consider
how the reachable set from any point is related to the promixity and/or shape
of the boundary. Likewise, even though the penny is controllable, optimal control might \textit{require} impacts.

The other object of future study is to understand when and if nonholonomic billiards are chaotic / ergodic / measure preserving. This is not immediately clear because purely continuous nonholonomic systems may not preserve measure even though holonomic systems always do. However,  if the system is ergodic, then there is a plethora of dense orbits which leads to interesting control questions.
\bibliographystyle{ieeetr}
\bibliography{references}

\end{document}